\documentclass{article}
\usepackage{amssymb, amsmath, mathrsfs, amsthm,verbatim}
\usepackage[nobysame]{amsrefs}
\usepackage[notcite,notref]{}
\usepackage{ulem}
\usepackage{graphicx}
\usepackage{color}
\usepackage[margin=1in]{geometry}
\usepackage{float, caption, subcaption}
\usepackage{algorithm}
\usepackage[noend]{algpseudocode}
\usepackage{hyperref,verbatim}

\usepackage{tikz}
\usetikzlibrary{shapes,decorations,calc,fit,positioning}
\tikzset{
    auto,node distance =1 cm and 1 cm,semithick,
    state/.style ={ellipse, draw, minimum width = 0.7 cm}
}

\newcommand{\pp}{\mathcal{P}}
\DeclareMathOperator{\obs}{\textsc{Obs}}
\newcommand{\meow}[4]{\left\langle\mathcal{H}^{#4}\oplus\mathcal{E}^{{#1}-{#4}}\right\rangle^{#3}_{{#1},{#2}-1}}

\newtheorem{notation}{Notation}

\newtheorem{theorem}{Theorem}
\theoremstyle{definition}
\newtheorem{definition}{Definition}
\newtheorem{remark}{Remark}
\newtheorem{example}{Example}
\newtheorem{question}{Question}
\newtheorem{lemma}{Lemma}
\newtheorem{corollary}{Corollary}
\newtheorem{proposition}{Proposition}
\newtheorem{observation}{Observation}

\definecolor{kgreen}{HTML}{3BAA1B}

\begin{document}

\title{Counting Power Domination Sets in Complete $m$-ary Trees}

\author{Sviatlana Kniahnitskaya        \and
        Michele Ortiz    \and   Olivia Ramirez \and Katharine Shultis \and
        Hays Whitlatch
}

\maketitle

\begin{abstract}
Motivated by the question of computing the probability of successful power domination by placing $k$ monitors uniformly at random, in this paper we give a recursive formula to count the number of power domination sets of size $k$ in a labeled complete $m$-ary tree.  As a corollary we show that the desired probability can be computed in exponential with linear exponent time. 
\end{abstract}

\section{Introduction}\label{intro}

The study of power domination sets arises from the monitoring of electrical network using Phase Measurement Units (PMUs or monitors).  This problem was first studied in terms of graphs in \cite{haynes2002domination} in 2002 and has been a topic of much interest since then (see e.g. \cite{baldwin1993power, benson2018zero, brueni2005pmu, ferrero2011power, zhao2006power}).  A PMU placed at a network node measures the voltage at the node and all current phasors at the node \cite{baldwin1993power}, and subsequently measures the voltage at some neighboring nodes using the propagation rules described in Definition \ref{PowerDomination}.  Since PMUs are expensive, it is desirable to find the minimum number of PMUs needed to monitor a network.  This problem is known to be to be NP-complete even for planar bipartite graphs (\cite{brueni2005pmu}). Since the cost of technology typically decreases but the cost of employment increases, it is feasible that the cost of placing extra PMUs is preferred to the cost of determining the minimum number of PMUs and an optimal placement.  Thus, in this paper, we begin to investigate how probable it is that a randomly placed set of $k$ PMUs will monitor a network.
\subsection{Terminology}
Let $T$ be a tree.  A \textit{rooted tree} is a tree in which one vertex has been designated the \textit{root}, and denoted $r(T)$, or simply $r$ when $T$ is clear from context. In a rooted tree, the \textit{parent} of a vertex $v$ is the vertex connected to $v$ on the path to the root. Since every vertex has a unique path to the root, every vertex other than the root has a unique parent. The root has no parent. A \textit{child} of a vertex $v$ is any vertex $w$ for which $v$ is the parent of $w$. A \textit{descendant} of a vertex $v$ is any vertex which is either the child of $v$ or is, recursively, the descendant of any of the children of $v$.

The \textit{height} of a vertex, $v$, in a rooted tree is the length of the longest path, not including $r$ as an internal vertex, to a leaf from that vertex, denoted $h(v)$. The height of the tree is the height of the root.

The \textit{complete $m$-ary tree of height $h$} is the tree of height $h$ satisfying that each internal vertex has $m$ children.  Throughout the literature this tree is also referred to as a full $m$-ary tree or a perfect $m$-ary tree.  We denote by $T_{m,h}$ the complete $m$-ary of height $h$ rooted at the center-most vertex, $r$.

For the purpose of counting power domination sets in a complete $m$-ary tree we will introduce a  new concept.  The \textit{extended $m$-ary tree of height $h$}, denoted $T^+_{m,h}$, is formed by adding an additional vertex, $r'$, and edge $\{r,r'
\}$ to the root to the tree $T_{m,h}$. That is, $T^+_{m,h}$ has vertex set $V=V(T_{m,h})\cup \{r'\}$ and edge set $E=E(T_{m,h})\cup \{\{r,r'\}\}$.  We refer to the added vertex $r'$ as the \textit{stem} of the tree. 

Suppose $G$ is a complete $m$-ary tree or an extended $m$-ary tree.  If $r_i$ is a child of the root $r$ we let $V_i$ be the descendants of $r_i$  and let $G_i$ be the induced subgraph of $G$ on $V_i\cup\{r_i,r\}$, that is $G_i=G[V_i\cup\{r_i,r\}]$.  Observe that $G_i$ is an extended $m$-ary tree of height $h-1$ with root $r_i$ and stem $r$.  In our proofs we will label the children of $r$ by $r_1, r_2,\ldots, r_m$ and refer to $G_i$ as the \textit{$i^\textrm{th}$ extended subtree of $G$}.

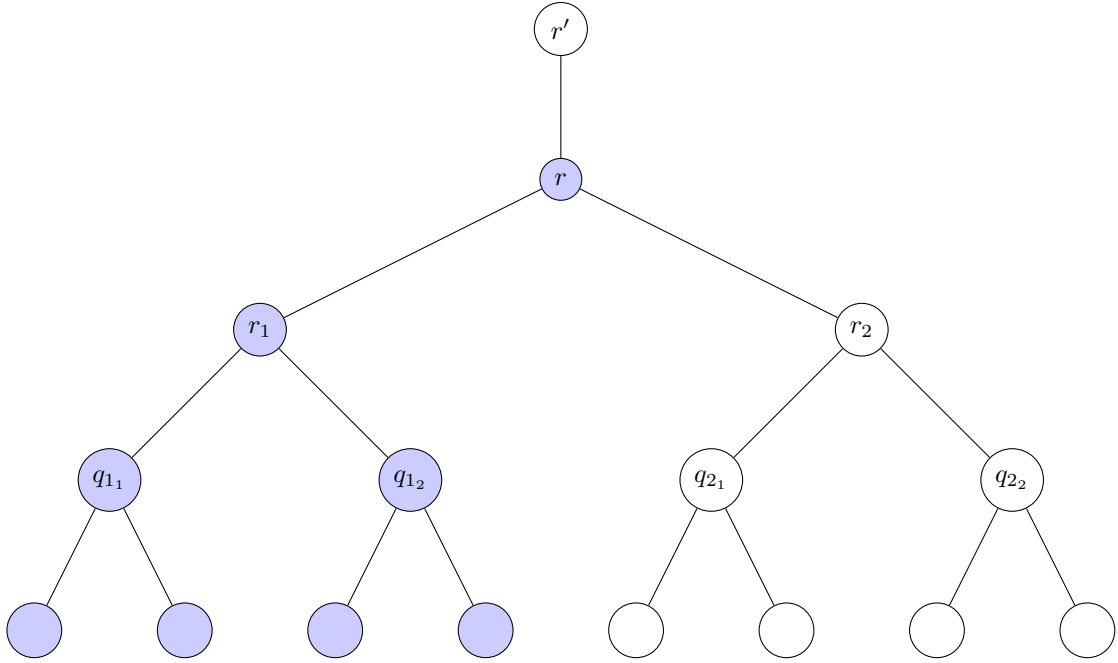
\begin{figure}[ht]
    \centering
    \begin{tikzpicture}[auto=center,every node/.style={circle,fill=blue!20}]
        \node[fill=white, draw] (r') at (7,8) {$r'$};
        \node[draw] (r) at (7,6) {$r$};
        \node[draw] (r1) at (3,4) {$r_1$};
        \node[fill=white, draw] (r2) at (11,4) {$r_2$};
        \node[draw] (q1) at (1,2) {$q_{1_1}$};
        \node[draw] (q2) at (5,2) {$q_{1_2}$};
        \node[fill=white, draw] (q3) at (9,2) {$q_{2_1}$};
        \node[fill=white, draw] (q4) at (13,2) {$q_{2_2}$};
        \node[draw] (p1) at (0,0) {\phantom{$p_1$}};
        \node[draw] (p2) at (2,0) {\phantom{$p_2$}};
        \node[draw] (p3) at (4,0) {\phantom{$p_3$}};
        \node[draw] (p4) at (6,0) {\phantom{$p_4$}};
        \node[fill=white, draw] (p5) at (8,0) {\phantom{$p_5$}};
        \node[fill=white, draw] (p6) at (10,0) {\phantom{$p_6$}};
        \node[fill=white, draw] (p7) at (12,0) {\phantom{$p_7$}};
        \node[fill=white, draw] (p8) at (14,0) {\phantom{$p_8$}};
        \draw (r') -- (r);
        \draw (r) -- (r1);
        \draw (r) -- (r2);
        \draw (r1) -- (q1);
        \draw (r1) -- (q2);
        \draw (r2) -- (q3);
        \draw (r2) -- (q4);
        \draw (q1) -- (p1);
        \draw (q1) -- (p2);
        \draw (q2) -- (p3);
        \draw (q2) -- (p4);
        \draw (q3) -- (p5);
        \draw (q3) -- (p6);
        \draw (q4) -- (p7);
        \draw (q4) -- (p8);
    \end{tikzpicture}
    \caption{$T^+_{2,3}$: Extended binary tree of height 3 with highlighted $1^{\textrm{st}}$ extended subtree of $T^+_{2,3}$}
    \label{fig:t23+}
\end{figure}

For any positive integer $r$, we will use $[r]$ to denote the set $\{1,2,\ldots,r\}$, and for any positive integers $q,r$ with $q\leq r$, we will use $[q,r]$ to denote the set $\{q,q+1,\ldots,r\}.$


\subsection{Power Domination}\label{PowerDom}

\begin{definition}[Graph Power  Domination]\label{PowerDomination}
Let $G=(V,E)$ and $S\subseteq V$.  Set $\pp^0(S)=N[S]$ (the closed neighborhood of $S$) and for $k\geq 1$ we let $\pp^k(S)=\pp^{k-1}(S)\cup N^*\left(\pp^{k-1}(S)\right)$ where $$N^*\left(\pp^{k-1}(S)\right)= \bigcup\limits_{v\in \pp^{k-1}(S)}\left\{x\in V :  N_G(v)\setminus \pp^{k-1}(S)=\{x\}\right\}$$
That is, $x\in N^*(A)$ if there exists some $a\in A$ such that $x$ is the only neighbor of $a$ not in $A$.  For a finite graph we see that eventually $\pp^k(S)=\pp^{k+1}(S)$ and we denote this by $\pp^{\infty}(S)$.  If $\pp^{\infty}(S)=V$ then we say that $S$ is a \textit{power dominating set} for $G$.  The minimum cardinality of a power dominating set for $G$ is referred to as the \textit{power dominating number} of $G$ and is denoted by $\gamma_P(G)$.
\end{definition}
\begin{definition}[Forcing Pairs]
Let $G=(V,E)$ and $S\subseteq V$ be given.  A pair $(x,y)$ is a \textit{forcing pair} for $S$ in $G$ if for some $k\geq 0$, $x\in \mathcal{P}^k(S)$ and $N_G(x)\setminus \mathcal{P}^k(S)=\{y\}$.  We denote this by $x \xrightarrow{G,S} y$, or simply by $x \xrightarrow{} y$ when $G$ and $S$ are clear from context. We also say \textit{$x$ forces $y$}.  The ordered set $(x_1, x_2, \ldots, x_k)$ is a \textit{forcing chain} for $S$ in $G$ if $x_i\xrightarrow{G,S} x_{i+1}$ for all $i\in [k-1]$.  This may be denoted by  $x_1 \xrightarrow{G,S} x_2 \xrightarrow{G,S}\cdots \xrightarrow{G,S} x_k$, dropping the superscripts when clear by context.
\end{definition}

\begin{remark}
Note that it is possible for a vertex to be forced by more than one vertex. Let $G$ be the graph with three vertices, $x$, $x'$, and $y$, and two edges, namely $\{x,y\}$ and $\{x',y\}$. If $S=\{x,x'\}$, then we have both $x\xrightarrow{G,S}y$ and $x'\xrightarrow{G,S}y$.
\end{remark}

\section{Power Domination in Complete $m$-ary Trees}\label{powerdom}
The following terminology will be central to our counting argument. 

\begin{definition}
Let $m\geq 2$, $h\geq 0$ and $G=T^{+}_{m,h}$. We say $S\subseteq V(G)\setminus\{r'\}$ is 
\begin{itemize}
    \item Type I if $S$ is a power dominating set for $G$;
    \item Type II if $S$ is not a power dominating set for $G$, but $S\cup\{r'\}$ is a power dominating set for $G$;
    \item Type 0 otherwise.
\end{itemize} 
We let $E_{h,m}^k$ denote the number of Type I sets of size $k$ that exist for $T^{+}_{m,h}$ and $H_{h,m}^k$ denote the number of Type II sets of size $k$ that exist for $T^{+}_{m,h}$.
\end{definition}
The letter $H$ is used to remind us that the set in question needs \textit{help} from $r'$ to successfully power-dominate and the letter $E$ is used to remind us that the set in question will successfully power-dominate $G-r'$ then \textit{exit} to assist in power-dominating $r'$.

\begin{notation}\label{children}
Let $G$ be an $m$-ary extended tree, with root $r$, stem $r'$, and extended subtrees $\{G_i\}_{i\in [d]}$, and let $S\subseteq V(G)\setminus\{r'\}$. 
We let \textsc{Obs}($G,S$) denote the set of vertices observed in the (attempted) power domination of $G$ by $S$. 

\end{notation}

\begin{proposition}\label{ZerosPersist}
Let $G$ be an extended $m$-ary tree and let $S\subseteq V(G)\setminus\{r'\}$.  If $S^+=S\cup\{r'\}$ is a power dominating set for $G$ then $S_i^+=(S\cap V(G_i))\cup\{r\}$ is a power dominating set for $G_i$ for all $i\in [m]$.
\end{proposition}

\begin{proof}
Choose and fix $S\subseteq V(G)\setminus\{r'\}$ and $i\in [m]$.  Since the neighborhoods of the sets $\{r\}$ and $\{r,r'\}$ are the same, it follows that $\textsc{Obs}(G, S\cup\{r, r'\})=\textsc{Obs}(G, S\cup\{r\}).$ Observe that the order in which power-domination occurs in the graph does not affect $\pp^{\infty}(S)$ and therefore we may assume that  $\textsc{Obs}(G, S^+)$ is computed by first exhausting power-domination moves resulting from the vertices of $G_i$.

Assume, by way of contradiction, that $\obs(G,S^+)=V(G)$ but $\obs(G_i, S_i^+)\neq V(G_i)$.  Choose and fix $x\in V(G_i)\setminus \obs(G_i, S_i^+)$ such that the distance from $x$ to $r$ is minimized. In particular the internal vertices of the path from $r$ to $x$ are all contained within $\obs(G_i, S_i^+)$.  Since $x\in \obs(G,S^+)$ there is some $y \in V(G)$ such that $y\xrightarrow{G,S^+} x$. Now $x\notin \{r,r_i\}$ since $r\in S_i^+$ and $r_i\in N_{G_i}(r)$. It follows that $y\in V(G_i)\setminus\{r\}$.

It must be the case that $y$ is either a child or the parent of $x$. We will show that $y$ is the parent of $x$. Assume to the contrary that $y$ is a child of $x$. Note that in $G_i$, with initial set $S_i^+$, the vertex $y$ doesn't force $x$. However, we do have that $y\xrightarrow{G,S^+} x$, and so there is a forcing chain $C=(v_1,v_2,\ldots,v_j,y)$ that starts outside of $V(G_i)$ and ends at $y$ that allows $y\xrightarrow{G,S^+} x $.  Since each of these paths must go through $x$ we have that $v_i=x$ for some $i\leq j$ and therefore  $x\xrightarrow{G,S^+} y$.  This would contradict that  $y\xrightarrow{G,S^+} x$ so we must proceed under the assumption that $y$ is the parent of $x$. 

By assumption, all of the ancestors of $x$ in $G_i$ are contained within $\obs(G_i, S_i^+)$.  Since $x\notin\obs(G_i,S_i^+)$, it follows that $y$ must have a second child, $x'$,  satisfying $x'\in V(G_i)\setminus \obs(G_i, S_i^+)$.  Hence $y \xrightarrow{G,S^+} x$ only after $x'$ is observed. The only way this can occur is if a child of $x'$ forces $x'$; however, the argument that $x$ must be forced by a parent and not a child could also be made for $x'$. This is a contradiction.
\end{proof}

\begin{notation}
Let $\left\langle\mathcal{H}^{\ell}\oplus \mathcal{E}^{m-\ell}\right\rangle_{m,h}^{k}=$
$$\sum\limits_{\substack{\substack{i_1\leq i_2 \leq\cdots\leq i_\ell\\ i_{\ell+1} \leq i_{\ell+2} \cdots\leq i_m}\\ i_1 + i_2+ \cdots+ i_m=k}}
 \binom{m}{\ell}\binom{\ell}{s_0, s_1, \ldots, s_k}\binom{m-\ell}{t_0, t_1, \ldots, t_k}\left(\prod\limits_{1\leq j\leq \ell} H_{m,h}^{i_j}\right)\left(\prod\limits_{\ell< j\leq m} E_{m,h}^{i_j}\right)$$
where $s_\alpha=\vert\{j\in [\ell] : i_j=\alpha\}\vert$ and $t_\beta=\vert\{j\in [\ell+1, m]: i_j=\beta\}\vert$.
\end{notation}

\begin{lemma}\label{biglemma}
Let $m\geq 2$, $h\geq 0$ and    $G=T^{+}_{m,h+1}$.  Then $\left\langle\mathcal{H}^{\ell}\oplus \mathcal{E}^{m-\ell}\right\rangle_{m,h}^{k}$ counts the number of ways to assign $k$ monitors to $V(G)\setminus\{r,r'\}$ so that exactly $\ell$ of the extended subtrees of $G$ receive Type II sets and the remaining $m-\ell$ extended subtrees receive Type I sets.
\end{lemma}
\begin{proof}
Let $m\geq 2$, $h\geq 0$ and    $G=T^{+}_{m,h+1}$.  
Choose and fix $i_1, i_2, \ldots, i_m$ so that, for some $\ell$ satisfying $0\leq \ell\leq m$, we have that $i_1\leq i_2 \leq\cdots\leq i_\ell$,   $i_{\ell+1} \leq i_{\ell+2} \cdots\leq i_m$,
and  $\sum_{j\in [m]}i_j=k$.  
Since $H_{m,h}^{i_j}$ and $E_{m,h}^{i_j}$ count the number of Type II and Type I (respectively) of size $i_j$ for $T^+_{m,h}$ then $$\left(\prod\limits_{1\leq j\leq \ell} H_{m,h}^{i_j}\right)\left(\prod\limits_{\ell< j\leq m} E_{m,h}^{i_j}\right)$$ counts the number of ways to assign $k$ monitors to $V(G)\setminus\{r,r'\}$ so that the $j^{\textrm{th}}$ subtree receives $i_j$ monitors that form a Type II set if $j\in [\ell]$ and form a Type I set otherwise. 
Now $$\binom{m}{\ell}\left(\prod\limits_{1\leq j\leq \ell} H_{m,h}^{i_j}\right)\left(\prod\limits_{\ell< j\leq m} E_{m,h}^{i_j}\right)$$ further allows for the Type II and Type I sets to be shuffled together with the restriction that the size of the Type II sets still appear in non-decreasing order relative to each other and similarly for the Type I sets. Multiplying by $\binom{\ell}{s_0, s_1, \ldots, s_k}$ and $\binom{m-\ell}{t_0, t_1, \ldots, t_k}$ removes these non-decreasing restrictions. 
Up to now, the count has been made for a fixed choice of $(i_1, i_2, \ldots, i_m)$. Thus by summing over all choices of $(i_1, i_2, \ldots, i_m)$ restricted to $i_1\leq i_2 \leq\cdots\leq i_\ell$,   $i_{\ell+1} \leq i_{\ell+2} \cdots\leq i_m$,
and  $\sum_{j\in [m]}i_j=k$ we count the number of ways to assign $k$ monitors to $V(G)\setminus\{r,r'\}$ so that exactly $\ell$ of the extended subtrees of $G$ receive Type II sets and the remaining $m-\ell$ extended subtrees receive Type I sets.
\end{proof}

\begin{observation}
For all $m\geq 2$:
$$ H_{m,0}^k=
\begin{cases}
1, & k=0\\
0, &\textrm{ otherwise}\\ 
\end{cases}
\quad\textrm{and}\quad E_{m,0}^k=
\begin{cases}
1, & k=1\\
0, &\textrm{ otherwise.}\\ 
\end{cases}
$$
\end{observation}
\begin{proof}
Since $V(G)\setminus\{r'\}=\{r\}$ we only have two sets to consider:  $S_1=\{r\}$ and $S_2=\emptyset$. For $S_1$, we have $k=1$, and $\textsc{Obs}(G,S_1)=V(G)$, yielding the values $E_{m,0}^1=1$ and $H_{m,0}^1=0$. For $S_2$, we have $k=0$, 
$\textsc{Obs}(G,S_2)=\emptyset$, and $\textsc{Obs}(G,S_2\cup\{r'\})=V(G)$, yielding the values $H_{m,0}^0=1$ and $E_{m,0}^0=0$.
\end{proof}
\begin{lemma}\label{height1}
For all $m\geq 2$:
$$ H_{m,1}^k=
\begin{cases}
m, & k=m-1\\
0, &\textrm{ otherwise}\\ 
\end{cases}
\quad\textrm{and}\quad E_{m,1}^k=
\begin{cases}
m+1, & k=m\\
\binom{m}{k-1}, & \textrm{ otherwise.}\\ 
\end{cases}
$$
\end{lemma}
\begin{proof}
Label the vertices so that the parent of $r$ is $r'$ and the children of $r$ are $\{r_1,\ldots,r_m\}$.  Then $G=(\{r',r, r_1,\ldots,r_m\},\{rr', rr_1,\ldots,rr_m\})$.   Fix $k \in \mathbb{Z}$.  Observe that if  $k\notin [0,m+1]$ then $H_{m,1}^k=E_{m,1}^k=0$ since it would be an impossible selection. As this agrees with our claim, we will proceed under the assumption that $k\in [0,m+1]$.  Let $S\subseteq V(G)\setminus \{r'\}$ of size $k$.  Assume first that $r\in S$; there are $\binom{m}{k-1}$ such sets. Since $N[S]\supseteq N[\{r\}]=V(G)$  we may conclude that $S$ is a Type I set.  If instead, we have $r\notin S$, then necessarily $k\in [0,m]$.  If  $0\leq k \leq m-2$ then at least two children of $r$ are not in $S$; call them $r_i$ and $r_j$.  Then $N[S]\subseteq N[S\cup \{r'\}]\subseteq V(G)\setminus\{r_i,r_j\}$ and therefore $r$ cannot force either vertex even if $\{r'\}$ is added to $S$. So in this case $S$ is Type 0.  If $k=m-1$ then $S=\{r_1, r_2, \ldots,r_m\}\setminus\{r_i\}$ for some $i$; there are $\binom{m}{m-1}=m$ such sets. Here, $N[S]= V(G)\setminus\{r',r_i\}$ and therefore $r$ cannot force either vertex. However, $N[S\cup\{r'\}]= V(G)\setminus\{r_i\}$ so $r\xrightarrow{G,S\cup\{r'\}}r_i$ and therefore $S$ is Type II.  Finally, if $k=m$ (and $r\notin S$) then there is only $\binom{m}{m}=1$ choice for $S$, namely $S=V(G)\setminus\{r',r\}$.  In this case $N[S]=V(G)\setminus\{r'\}$ and $r$ forces $r'$ so it is a Type I set.  The result follows. 
\end{proof}
\begin{corollary}
Let $m\geq 2$ and $G=T^+_{m,1}$.
If $S$ is a Type II set then 
$$N_G[r]\setminus\textsc{Obs}(G,S)=
\{r',r_i\} \textrm{ for some $i\in [m]$,}$$
Furthermore,
$$H_{m,1}^k=
\left\langle\mathcal{H}^{1}\oplus \mathcal{E}^{m-1}\right\rangle_{m,0}^{k}$$
and 
$$E_{m,1}^k=
\left\langle\mathcal{H}^{0}\oplus \mathcal{E}^{m}\right\rangle_{m,0}^{k}+\sum\limits_{\ell=0}^{m}\left\langle\mathcal{H}^{\ell}\oplus \mathcal{E}^{m-\ell}\right\rangle_{m,0}^{k-1}.$$
\end{corollary}
\begin{proof}
The only case where $S$ is Type II is when $S=\{r_1, r_2, \ldots, r_m\}\setminus\{r_i\}$ for some $i\in [m]$.  It follows from the proof that if $S$ is a Type II set then 
$$N_G[r]\setminus\textsc{Obs}(G,S)=
\{r',r_i\} \textrm{ for some $i\in [m]$,}$$
Observe that $V(G)\setminus\{r,r'\}=\{r_1, r_2, \ldots, r_m\}$. Now $\left\langle\mathcal{H}^{\ell}\oplus \mathcal{E}^{m-\ell}\right\rangle_{m,0}^{X}$ counts the number of ways to assign $X$ monitors  to $V(G)\setminus\{r,r'\}$ so that exactly $\ell$ subtrees receive a Type II set and the remaining $m-\ell$ receive Type I sets.  In this context $\left\langle\mathcal{H}^{\ell}\oplus \mathcal{E}^{m-\ell}\right\rangle_{m,0}^{X}$ counts the number of ways to place monitors on exactly $\ell$ of the children of $r$ and therefore $\left\langle\mathcal{H}^{\ell}\oplus \mathcal{E}^{m-\ell}\right\rangle_{m,0}^{X}=0$ unless $X=m-\ell$ in which case $\left\langle\mathcal{H}^{\ell}\oplus \mathcal{E}^{m-\ell}\right\rangle_{m,0}^{X}=\binom{m}{m-\ell}$.

It follows that $$ \left\langle\mathcal{H}^{1}\oplus \mathcal{E}^{m-1}\right\rangle_{m,0}^{k}=
\begin{cases}
\binom{m}{m-1}, & k=m-1\\
0, &\textrm{ otherwise}\\ 
\end{cases}$$
$$\left\langle\mathcal{H}^{0}\oplus \mathcal{E}^{m}\right\rangle_{m,0}^{k}=
\begin{cases}
\binom{m}{m}, & k=m\\
0, & \textrm{otherwise}\\
\end{cases}
$$
$$\sum\limits_{\ell=0}^{m}\left\langle\mathcal{H}^{\ell}\oplus \mathcal{E}^{m-\ell}\right\rangle_{m,0}^{k-1}=\left\langle\mathcal{H}^{\ell}\oplus \mathcal{E}^{k-1}\right\rangle_{m,0}^{k-1}=
\binom{m}{k-1}.
$$
\end{proof}

\begin{lemma}\label{height2}
For all $m\geq 2$,
$$ H_{m,2}^k=\left\langle\mathcal{H}^{m}\oplus \mathcal{E}^{0}\right\rangle_{m,1}^{k}
\quad\textrm{and}\quad E_{m,2}^k=\sum\limits_{\ell=0}^{m-1}\left\langle\mathcal{H}^{\ell}\oplus \mathcal{E}^{m-\ell}\right\rangle_{m,1}^{k}+\sum\limits_{\ell=0}^{m}\left\langle\mathcal{H}^{\ell}\oplus \mathcal{E}^{m-\ell}\right\rangle_{m,1}^{k-1}
.$$
\end{lemma}
\begin{proof}
Let $G=T^+_{m,2}$ with root $r$ and stem $r'$.  Label the children of $r$ by $r_1, \ldots, r_m$ and for $i\in [m]$ label the children of $r_i$ by $q_{i_1},\ldots, q_{i_m}$.  The $i^{\textrm{th}}$ extended subtree of $G$ is $G_i$ with vertex set $\{r,r_i,q_{i_1},\ldots, q_{i_m} \}$ (see figure below). Let $S\subseteq V(G)\setminus\{r'\}$ and for each $i\in [m]$ let $S_i=(S\cap V(G_i))\setminus\{r\}$.  By Proposition \ref{ZerosPersist} if $S_i$ is Type 0 for any $i$ then $S$ does not contribute to $H^{k}_{m.2}$ or $E^{k}_{m,2}$, thus we only consider the cases where all the $S_i$'s are Type I or Type II.

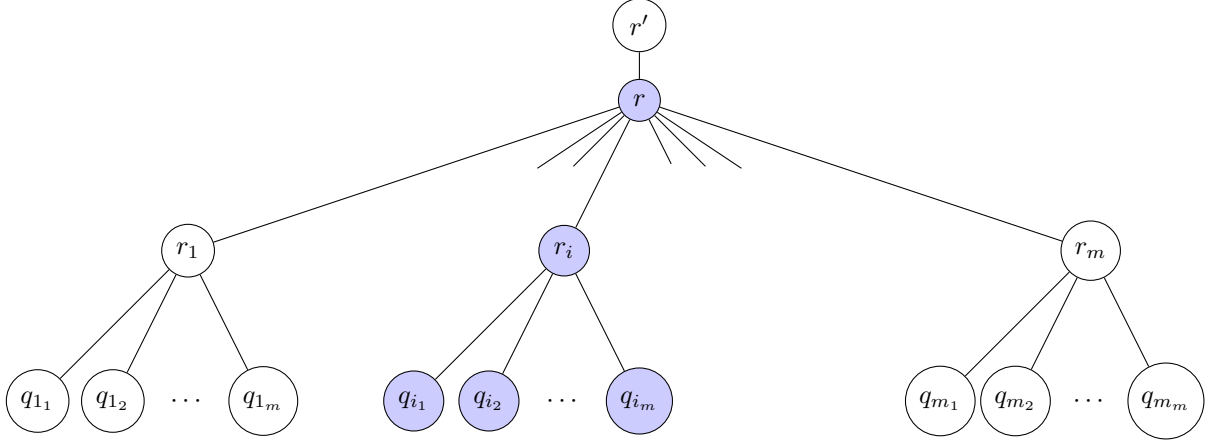
\begin{tikzpicture}[auto=center,every node/.style={circle,fill=white}]
        \node[draw] (r') at (7,7) {$r'$};
        \node[draw,fill=blue!20] (r) at (7,6) {$r$};
        \node[draw] (r1) at (1,4) {$r_1$};
        \node[draw,fill=blue!20] (ri) at (6,4) {$r_i$};
        \node (r3) at (5.5,5) {};
        \node (r4) at (6,5) {};
        \node (r5) at (7.5,5) {};
        \node (r6) at (8,5) {};
        \node (r7) at (8.5,5) {};
        \node[draw] (rm) at (13,4) {$r_m$};
        \node[draw] (q11) at (-1,2) {$q_{1_1}$};
        \node[draw] (q12) at (0,2) {$q_{1_2}$};
        \node (q13) at (1,2) {$\cdots$};
        \node[draw] (q1m) at (2,2) {$q_{1_m}$};
        \node[draw,fill=blue!20] (qi1) at (4,2) {$q_{i_1}$};
        \node[draw, fill=blue!20] (qi2) at (5,2) {$q_{i_2}$};
        \node (qi3) at (6,2) {$\cdots$};
        \node[draw, fill=blue!20] (qim) at (7,2) {$q_{i_m}$};
        \node[draw] (qm1) at (11,2) {$q_{m_1}$};
        \node[draw] (qm2) at (12,2) {$q_{m_2}$};
        \node (qm3) at (13,2) {$\cdots$};
        \node[draw] (qmm) at (14,2) {$q_{m_m}$};
        \draw (r') -- (r);
        \draw (r) -- (r1);
        \draw (r) -- (ri);
        \draw (r) -- (r3);
        \draw (r) -- (r4);
        \draw (r) -- (r5);
        \draw (r) -- (r6);
        \draw (r) -- (r7);
        \draw (r) -- (rm);
        \draw (r1) -- (q11);
        \draw (r1) -- (q12);
        \draw (r1) -- (q1m);
        \draw (ri) -- (qi1);
        \draw (ri) -- (qi2);
        \draw (ri) -- (qim);
        \draw (rm) -- (qm1);
        \draw (rm) -- (qm2);
        \draw (rm) -- (qmm);
    \end{tikzpicture}
    
Case 1: Suppose $r\notin S$ and  $S_i$ is Type II for $G_i$ for each $i\in [m]$.  By Lemma \ref{biglemma}, with $\ell=m$ and $h=1$, there are $\left\langle\mathcal{H}^{m}\oplus \mathcal{E}^{0}\right\rangle_{m,1}^{k}$ ways to choose $S$. It follows from the proof of Lemma \ref{height1} that for each $i\in [m]$, $S_i=\{q_{i_1}, q_{i_2}, \ldots,q_{i_m}\}\setminus\{q_{i_{j_i}}\}$ for some $j_i\in [m]$. Therefore $$S=\bigcup_{i\in [m]}\left(\{q_{i_1}, q_{i_2}, \ldots,q_{i_m}\}\setminus\{q_{i_{j_i}}\}\right)$$ which in turn implies that $r\notin \textsc{Obs}(G,S)$ since for any $i\in[m]$, $r_i$ cannot force both $q_{i_j}$ and $r$. On the other hand, since $\{r,r'\} \subseteq N_G[r]=\{r',r, r_1,\ldots,r_m\}$ we have that $r_i\xrightarrow{G,S\cup\{r\}}q_{i_{j_i}}$ for each $i\in [m]$ so that $S\cup\{r\}$ is a dominating set for $G$. We conclude that whenever $r\not\in S$ and $S_i$ is Type II for $S_i$ for each $i\in[m]$, then $S$ is Type II. Moreover, there are $\langle\mathcal{H}^{m}\oplus\mathcal{E}^{0}\rangle_{m,1}^k$ such sets, $S$. 

Case 2: Suppose $r\notin S$ and for some $\ell\in[0,m-1]$, we have $S_i$ is Type II for the corresponding $G_i$ for $\ell$ of the sets $G_i$ and the remaining $m-\ell$ sets $S_i$ are Type I for their corresponding $G_i$. Observe that by Lemma \ref{biglemma} there are $\sum_{\ell=0}^{m-1}\left\langle\mathcal{H}^{\ell}\oplus \mathcal{E}^{m-\ell}\right\rangle_{m,1}^{k}$ such ways to choose $S$. The restriction of $S$ to the the extended subtrees of $G$ results in $\ell$ Type II sets and $m-\ell>0$ Type I sets. Without loss of generality, we may assume $S_i$ is Type II for $G_i$ for all $i\in[\ell]$ and $S_j$ is Type I for $G_j$ for all $j\in[\ell+1,m]$. Appealing to the proof of Lemma \ref{height1} (as we did in Case 1) we may conclude that 
\begin{align*}
    \textsc{Obs}(G,S) &\supseteq \bigcup_{i\in [m]}\textsc{Obs}(G,S_i)\\
    &=\left(\bigcup_{i\in [\ell]}\textsc{Obs}(G,S_i)\right)\cup\left(\bigcup_{i\in [\ell+1,m]}\textsc{Obs}(G,S_i)\right)\\
    &=\left(\bigcup_{i\in [\ell]}\{r_i,q_{i_1}, q_{i_2}, \ldots,q_{i_m}\}\setminus\{q_{i_{j_i}}\}\right)\cup\left(\bigcup_{i\in [\ell+1,m]}V(G_i)\right).
\end{align*}
It then follows that $r_m\xrightarrow{G,S}r\xrightarrow{G,S}r'$ and then for $i\in [\ell]$, $r_i\xrightarrow{G,S}q_{i_{j_i}}$.  Thus, these $\sum_{\ell=0}^{m-1}\langle\mathcal{H}^\ell\oplus\mathcal{E}^{m-\ell} \rangle_{m,1}^k$ sets are all Type I. 

Case 3: Suppose $r\in S$ and for some $\ell\in[0,m]$, we have $S_i$ is Type II for $G_i$ for $\ell$ of the sets $G_i$ and the remaining $m-\ell$ sets $S_i$ are Type I for their corresponding $G_i$. Without loss of generality, we may assume $S_i$ is Type II for $G_i$ for all $i\in[\ell]$ and $S_j$ is Type $I$ for $G_j$ for all $j\in[\ell+1,m]$. By Lemma \ref{biglemma}, and noting that there are $k-1$ monitors placed on vertices other than $r$, there are $\sum_{\ell=0}^{m}\left\langle\mathcal{H}^{\ell}\oplus \mathcal{E}^{m-\ell}\right\rangle_{m,1}^{k-1}$ ways to choose $S$.  Again by appealing to the proof of Lemma \ref{height1} we have that 
\begin{align*}
    \textsc{Obs}(G,S) &\supseteq N_G[r]\cup\left(\bigcup_{i\in [m]}\textsc{Obs}(G,S_i)\right)\\
    &\supseteq  N_G[r]\cup\left(\bigcup_{i\in [m]}\{q_{i_1}, q_{i_2}, \ldots,q_{i_m}\}\setminus\{q_{i_{j_i}}\}\right)\\
    &=V(G)\setminus\left(\bigcup_{i\in[m]}\{q_{i_{j_i}}\}\right).
\end{align*} It then follows that $r_i\xrightarrow{G,S}q_{i_{j_i}}$ for each $i\in [m]$.   We conclude that these $\sum_{\ell=0}^{m}\left\langle\mathcal{H}^{\ell}\oplus \mathcal{E}^{m-\ell}\right\rangle_{m,1}^{k-1}$ sets are all Type I.

\end{proof}
\begin{corollary}
Let $m\geq 2$ and $G=T^+_{m,2}$.
If $S$ is a Type II set then 
$$N_G[r]\setminus\textsc{Obs}(G,S)=
\{r',r\}.$$
Furthermore,
$$H_{m,2}^k=
\left\langle\mathcal{H}^{m}\oplus \mathcal{E}^{0}\right\rangle_{m,1}^{k}
$$
and 
$$E_{m,2}^k=
\sum\limits_{\ell=0}^{m-1}\left\langle\mathcal{H}^{\ell}\oplus \mathcal{E}^{m-\ell}\right\rangle_{m,1}^{k}+\sum\limits_{\ell=0}^{m}\left\langle\mathcal{H}^{\ell}\oplus \mathcal{E}^{m-\ell}\right\rangle_{m,1}^{k-1}.$$
\end{corollary}

In the proof of Lemma \ref{height2}, we twice claimed that \textit{without loss of generality} we may assume that if $\ell$ of the subtrees of $T^+$ are Type II then we may assume they are the first $\ell$ trees. We will use this without claim in the rest paper; in particular in the proofs of Lemma \ref{hugelemma} and Theorem \ref{bigtheorem}.

\begin{lemma}\label{hugelemma}
Let $m\geq 2$, $h\geq 1$, and $G=T^+_{m,h}$.
If $S$ is a Type II set then 
$$N_G[r]\setminus\textsc{Obs}(G,S)=
\begin{cases}
\{r',r_i\} \textrm{ for some $i\in [m]$,} &\textrm{if }h\textrm{ is odd}\\
\{r',r\}, & \textrm{if }h\textrm{ is even}\\
    \end{cases}.$$
Furthermore,
$$H_{m,h}^k=
\begin{cases}
\left\langle\mathcal{H}^{1}\oplus \mathcal{E}^{m-1}\right\rangle_{m,h-1}^{k}, & \textrm{if }h\textrm{ is odd}\\ 
\left\langle\mathcal{H}^{m}\oplus \mathcal{E}^{0}\right\rangle_{m,h-1}^{k}, &\textrm{if }h\textrm{ is even}\\
\end{cases}
$$
and 
$$E_{m,h}^k=
\begin{cases}
\left\langle\mathcal{H}^{0}\oplus \mathcal{E}^{m}\right\rangle_{m,h-1}^{k}+\sum\limits_{\ell=0}^{m}\left\langle\mathcal{H}^{\ell}\oplus \mathcal{E}^{m-\ell}\right\rangle_{m,h-1}^{k-1}, &\textrm{if } h\textrm{ is odd}\\
\sum\limits_{\ell=0}^{m-1}\left\langle\mathcal{H}^{\ell}\oplus \mathcal{E}^{m-\ell}\right\rangle_{m,h-1}^{k}+\sum\limits_{\ell=0}^{m}\left\langle\mathcal{H}^{\ell}\oplus \mathcal{E}^{m-\ell}\right\rangle_{m,h-1}^{k-1}, &\textrm{if } h\textrm{ is even}\\ 
\end{cases}.$$
\end{lemma}
\begin{proof}
The statement is true when $h\in \{1,2\}$ by the corollaries to Lemma \ref{height1} and \ref{height2}.
We proceed by strong induction on $h$ by assuming that the claim is true for $h=1,2,\ldots,n$.  Let $G=T^+_{m,h}$ and let $S\subseteq V(G)\setminus\{r'\}$ be a Type II set.

Case 1: Suppose $h=n+1$ is odd.   By Proposition \ref{ZerosPersist}, if for any $i\in[m]$ it is the case that $S_i$ is Type 0 for $G_i$, then $S$ is Type 0 for $G$.  Since $S$ is presumed to be a Type II set, we may proceed under the assumption $S_i$ is Type II for the first $\ell$ extended subtrees and Type I for the latter $m-\ell$ extended subtrees. By inductive hypothesis $N_G[r_i]\setminus \textsc{Obs}(G_i,S_i)=\{r,r_i\}$ for each $i\in [\ell]$.  

Subcase 1: Let $\ell=0$ and suppose $r\notin S$.  Since $S_i$ is Type I for each $i\in [m]$ we have 
\[\textsc{Obs}(G,S)\supseteq\bigcup_{i\in [m]}\textsc{Obs}(G_i,S_i)=\bigcup_{i\in [m]}V(G_i)=V(G)\setminus\{r'\}.\]
However $r\xrightarrow{G,S}r'$ so the $\left\langle\mathcal{H}^{0}\oplus \mathcal{E}^{m}\right\rangle_{m,h-1}^{k}$ sets in question are Type I. This gives the first summand for $E^k_{m,h}$ in the case that $h$ is odd. 

 Subcase 2: Let $\ell=1$ and suppose $r\notin S$.  By inductive hypothesis $r,r_1\notin \textsc{Obs}(G_1,S_1)$ and therefore $r',r_1\notin \bigcup_{i\in [m]}\textsc{Obs}(G_i,S_i)$. Since $r$ cannot force both $r'$ and $r_1$, it follows that $S$ is not Type I.  Observe that $N_G[r]\setminus\obs(G,S)=\{r',r_1\}$ which agrees with the first claim in the case where $h$ is odd. We show that the sets $S$ in this case are Type II. Assume, by way of contradiction, that there is some vertex $v\in V(G)$ with $v\notin\textsc{Obs}(G,S\cup\{r'\})$. Since $S_2, \ldots, S_m$ are Type I then $v\in V(G)\setminus\bigcup_{i\in [2,m]}V(G_i)$.  Clearly $r'\in \textsc{Obs}(G,S\cup\{r'\})$ so $v\in V(G_1)$.  Now $r'\xrightarrow{G,S\cup\{r'\}} r\xrightarrow{G,S\cup\{r'\}}r_1$ so $v\in V(G_1)\setminus\{r,r_1\}$.  However, $S_1$ is a Type II set so $r\xrightarrow{G_1,S_1}r_1\xrightarrow{G_1,S_1}\cdots \xrightarrow{G_1,S_1}v.$ But then $r\xrightarrow{G,S\cup\{r'\}}r_1\xrightarrow{G,S\cup\{r'\}}\cdots \xrightarrow{G,S\cup\{r'\}}v$ as well.  It follows that the $\left\langle\mathcal{H}^{1}\oplus \mathcal{E}^{m-1}\right\rangle_{m,h-1}^{k}$ sets in question are Type II. This gives the result for $H^k_{m,h}$ when $h$ is odd. 

Subcase 3: Let $2\leq\ell\leq m$ and suppose $r\notin S$. By inductive hypothesis $r_i\notin \textrm{Obs}(G_i,S_i)$ for all $i\in [\ell]$ and, since $\ell\geq 2$, $r$ cannot force all of the vertices $r_i$ with $i\in[\ell]$.   Thus $\textsc{Obs}(G,S)\subseteq V(G)\setminus\{r',r, r_1, \ldots,r_\ell\}$ so $S$ is not Type I.  Assume, by way of contradiction, that $S$ is Type II.  Then for each $v\in V(G)$ there is a chain $$r'\xrightarrow{G,S\cup\{r'\}}r\xrightarrow{G,S\cup\{r'\}}r_i\xrightarrow{G,S\cup\{r'\}}\cdots \xrightarrow{G,S\cup\{r'\}}v.$$ However, this is impossible since $r$ can only force a child once all of the other children have been observed. The sets in question are Type 0.

Subcase 4: Let $0\leq \ell\leq m$ and suppose $r\in S$. Note that \[\textsc{Obs}(G,S)\supseteq N_G[r]\cup\bigcup_{i\in [m]}\textsc{Obs}(G_i,S_i\cup \{r\})=N_G[r]\cup\bigcup_{i\in [m]}V(G_i)=V(G).\]  It follows that the $\sum\limits_{\ell=0}^{m}\left\langle\mathcal{H}^{\ell}\oplus \mathcal{E}^{m-\ell}\right\rangle_{m,h-1}^{k-1}$ sets in question are Type I. This gives the second summand for $E^k_{m,h}$ in the case where $h$ is odd.

Case 2: Suppose $h=n+1$ is even.   By Proposition \ref{ZerosPersist}, if for any $i\in[m]$ it is the case that $S_i$ is Type 0 (for $G_i$), then $S$ is Type 0 for $G$.  Assuming this is not the case, we may proceed under the assumption $S_i$ is Type II for the first $\ell$ extended subtrees and Type I for the latter $m-\ell$ extended subtrees.  By inductive hypothesis, for each $i\in[\ell]$, there is some $j_i\in[m]$ such that $N_G[r_i]\setminus \textsc{Obs}(G_i,S_i)=\{r,q_{i_{j_i}}\}$.  Without loss of generality we will assume that $N_G[r_i]\setminus \textsc{Obs}(G_i,S_i)=\{r,q_{i_{1}}\}$ for each $i\in [\ell]$. 

Subcase 1: Let $\ell=m$ and suppose $r\notin S$.  Since $S_i$ is Type II for each $i\in [m]$ we have $r, q_{i_{1}}\notin\textsc{Obs}(G_i,S_i)$.  Observe that this implies that $r \notin\textsc{Obs}(G,S)$ since otherwise $r_i\xrightarrow{G,S}r$ implies $r_i\xrightarrow{G_i,S_i}r$ for some $i\in [m]$.  Thus the sets in question are not Type I.  Observe that $N_G[r]\setminus\obs(G,S)=\{r,r'\}$ which agrees with the first claim in the case where $h$ is even.  For each $i\in [m]$, $\textsc{Obs}(G_i,S_i\cup\{r\})=V(G_i);$ however, $N_{G}[r]\setminus \left(\cup_{i\in[m]}\textsc{Obs}(G_i,S_i)\right)=\{r',r\}$ since by inductive hypothesis $N_{G_i}[r]\setminus \textsc{Obs}(G_i,S_i)=\{r\}$.  It follows that $V(G_i)=\textsc{Obs}(G_i,S_i\cup\{r\})=\textsc{Obs}(G_i,S_i\cup\{r'\})$ and therefore 
\begin{align*}\textsc{Obs}(G,S\cup\{r'\}) &\supseteq N_G[r']\cup\left(\bigcup_{i\in [m]}\textsc{Obs}(G_i,S_i\cup \{r'\})\right)\\
&=N_G[r]\cup\left(\bigcup_{i\in [m]}V(G_i)\right)\\&=V(G).
\end{align*}It follows that the $\left\langle\mathcal{H}^{m}\oplus \mathcal{E}^{0}\right\rangle_{m,h-1}^{k}$ sets in question are Type II. This gives the result for $H^k_{m,h}$ in the case that $h$ is even.

Subcase 2: Let $0\leq \ell\leq m-1$ and suppose $r\notin S$.  By inductive hypothesis $N_{G_i}(r_i)\setminus \textsc{Obs}(G_i,S_i)=\{r,q_{i_1}\}$ for each $i\in [\ell]$.  Importantly $\{r_1, r_2, \ldots, r_m\}\subset \textsc{Obs}(G,S)$ and since $S_m$ is Type I, then $r_m\xrightarrow{G_m,S_m}r$. Therefore $r_m\xrightarrow{G,S}r\xrightarrow{G,S}r'$ and subsequently, for each $i\in[\ell]$, $r_i\xrightarrow{G,S}q_{i_1}$ initiates a forcing chain equivalent to the the one in Subcase 1.   It follows that the $\sum\limits_{\ell=0}^{m-1}\left\langle\mathcal{H}^{\ell}\oplus \mathcal{E}^{m-\ell}\right\rangle_{m,h-1}^{k}$ sets in question are Type I. This gives the first summand for $E^k_{m,h}$ in the case that $h$ is even.

Subcase 3:  Let $0\leq \ell\leq m$ and suppose $r\in S$. For each $i\in[\ell]$, $S_i$ is Type II for $G_i$ and therefore $\textsc{Obs}(G_i,S_i\cup\{r\})=V(G_i)$.  It then follows that 
\begin{align*}
    \textsc{Obs}(G,S)
    &=\textsc{Obs}(G,S\cup\{r\})\\
    &=N_G[r]\cup\left(\bigcup_{i\in [m]}\textsc{Obs}(G_i,S_i\cup \{r\})\right)\\
    &=N_G[r]\cup\left(\bigcup_{i\in [m]}V(G_i)\right)\\
    &=V(G).
\end{align*}
It follows that the $\sum\limits_{\ell=0}^{m}\left\langle\mathcal{H}^{\ell}\oplus \mathcal{E}^{m-\ell}\right\rangle_{m,h-1}^{k-1}$ sets in question are Type I. This gives the second summand for $E^k_{m,h}$ in the case that $h$ is even.

\end{proof}
 
\begin{theorem}\label{bigtheorem}
Let $m\geq 2$ and $h\geq 2$.  The number of power dominating sets of size $k$ for $G=T_{m,h}$ is:

$$N(m,h,k)=\begin{cases}
\sum\limits_{\ell=0}^{m-1}\left
\langle \mathcal{H}^\ell\oplus\mathcal{E}^{m-\ell}
\right\rangle^{k}_{m,h-1}+\sum\limits_{\ell=0}^{m}\left
\langle \mathcal{H}^\ell\oplus\mathcal{E}^{m-\ell}
\right\rangle^{k-1}_{m,h-1}, &\textrm{ $h$ is even}\\
\sum\limits_{\ell=0}^{1}\left
\langle \mathcal{H}^\ell\oplus\mathcal{E}^{m-\ell}
\right\rangle^{k}_{m,h-1}+\sum\limits_{\ell=0}^{m}\left
\langle \mathcal{H}^\ell\oplus\mathcal{E}^{m-\ell}
\right\rangle^{k-1}_{m,h-1}, &\textrm{ $h$ is odd}.
\end{cases}$$
\end{theorem}
\begin{proof}
Let $G=T_{m,h}$ with root $r$ and root children $r_1, r_2, \ldots, r_m$.  For each $i\in[m]$, let $G_i$ be the extended subtree with stem $r$, root $r_i$ and containing all the descendants of $r_i$. For each $i\in[r]$ let the children of $r_i$ be $\{q_{i_j}\}_{j\in[m]}$.  Choose $S\subseteq V(G)$ and let $S_i=S\cap \left(V(G_i)\setminus\{r\}\right)$. By Proposition \ref{ZerosPersist} we need only consider cases where, for each $i\in[m]$, $S_i$ is a Type I or Type II set for $G_i$. Thus we will assume, without loss of generality, that the first $\ell$ sets are Type II and the remaining $m-\ell$ sets are Type I.

Case 1:  Suppose $r\in S$.  Note that $\textsc{Obs}(G,S)\supseteq\bigcup_{i\in [m]}\textsc{Obs}(G_i,S_i\cup\{r\})=V(G)$.  It follows that for any $0\leq \ell\leq m$, the set $S$ will be a power dominating set and the number of sets in question is $\sum\limits_{\ell=0}^{m}\left
\langle \mathcal{H}^\ell\oplus\mathcal{E}^{m-\ell}
\right\rangle^{k-1}_{m,h-1}$. This gives the second summand in both instances: where $h$ is even or $h$ is odd.

Case 2:  Suppose $h$ is even and $r\notin S$.   Observe that for each $i\in[\ell]$, we have $S_i$ is a Type II set for $G_i$, and consequently $r,q_{i_j} \notin \textsc{Obs}(G_i,S_i)$ for some $j\in [m]$.  Thus, if all the sets are Type II then no child of $r$ can force $r$ and $S$ is not a power dominating set.  On the other hand, if $\ell<m$ then $r_m\xrightarrow{G,S}r$ and for each $i\in [\ell]$, $r_i\xrightarrow{G,S}q_{i_j}$ initiates the same forcing chain that begins with  $r_i\xrightarrow{G_i,S_i\cup\{r\}}q_{i_j}$.  Hence, the $\sum\limits_{\ell=0}^{m-1}\left
\langle \mathcal{H}^\ell\oplus\mathcal{E}^{m-\ell}
\right\rangle^{k}_{m,h-1}$ sets in question are power dominating sets.

Case 3:  Suppose $h$ is odd and $r\notin S$.   Observe that for each Type II set $r,r_i \notin \textsc{Obs}(G_i,S_i)$.  Thus, if $\ell\geq 2$, the vertex $r$ cannot force both $r_1$ and $r_2$ and $S$ is not a power dominating set.  On the other hand, if $0\leq\ell\leq 1$ then $r_m\xrightarrow{G,S}r\xrightarrow{G,S}r_i$ for $i\in [\ell]$, initiating the same forcing chain that begins with  $r\xrightarrow{G_i,S_i\cup\{r\}}r_i$.  Hence, the $\sum\limits_{\ell=0}^{1}\left
\langle \mathcal{H}^\ell\oplus\mathcal{E}^{m-\ell}
\right\rangle^{k}_{m,h-1}$ sets in question are power dominating sets.

\end{proof}

\begin{corollary} Let $G=T_{m,h}$ with $m,h\geq 2$.
The probability that a uniformly at random selected subset of $V(G)$ of size $k$ is a power dominating set is $$p(m,h,k)=\frac{N(m,h,k)}{\binom{|V(G)|}{k}}$$ where $|V(G)|=\frac{m^{h+1}-1}{m-1}$.
\end{corollary}

\begin{example}
We will show that $N(2,2,4)=33$. Using Theorem \ref{bigtheorem} yields \begin{eqnarray}N(2,2,4)&=&\sum_{\ell=0}^{1}\meow{2}{2}{4}{\ell}+\sum_{\ell=0}^{2}\meow{2}{2}{3}{\ell} \nonumber\\
&=&\meow{2}{2}{4}{0}+\meow{2}{2}{4}{1}+\meow{2}{2}{3}{0}+\meow{2}{2}{3}{1}+\meow{2}{2}{3}{2}\nonumber.
\end{eqnarray}
Computing the first summand is done as follows:
\begin{eqnarray}
\left\langle\mathcal{H}^{0}\oplus \mathcal{E}^{2}\right\rangle_{2,1}^{4}&=&\sum\limits_{\substack{\substack{\\ i_{1} \leq i_2}\\ i_1 + i_2=4}}
 \binom{2}{0}\binom{0}{s_0, s_1, \ldots, s_4}\binom{2}{t_0, t_1, \ldots, t_4}\left(\prod\limits_{1\leq j\leq 0} H_{2,1}^{i_j}\right)\left(\prod\limits_{0< j\leq 2} E_{2,1}^{i_j}\right)\nonumber\\
 &=&\sum\limits_{\substack{\substack{\\ i_{1} \leq i_2}\\ i_1 + i_2=4}}
 \binom{2}{t_0, t_1, \ldots, t_4}\left(\prod\limits_{0< j\leq 2} E_{2,1}^{i_j}\right)\nonumber\\
 &=&  \binom{2}{1,0,0,0,1}\left(\prod\limits_{0< j\leq 2} E_{2,1}^{i_j}\right)+
 \binom{2}{0, 1,0,1,0}\left(\prod\limits_{0< j\leq 2} E_{2,1}^{i_j}\right)+
 \binom{2}{0,0,2,0,0}\left(\prod\limits_{0< j\leq 2} E_{2,1}^{i_j}\right) \nonumber\\
 &=&  2\left(E_{2,1}^{0}E_{2,1}^{4}\right)+
 2\left(E_{2,1}^{1}E_{2,1}^{3}\right)+
 1\left(E_{2,1}^{2}E_{2,1}^{2}\right)= 2\cdot 0\cdot 0 +
 2\cdot 1\cdot 1+
 1\cdot 3\cdot 3 =11\nonumber.
\end{eqnarray}
By similar computations we obtain the other four summands:
\begin{eqnarray}
\left\langle\mathcal{H}^{1}\oplus \mathcal{E}^{1}\right\rangle_{2,1}^{4} 
 &=& 
 2\left(H_{2,1}^{0}\right)\left(E_{2,1}^{4}\right)
+2\left(H_{2,1}^{1}\right)\left(E_{2,1}^{3}\right)
+2\left(H_{2,1}^{2}\right)\left(E_{2,1}^{2}\right)
+2\left(H_{2,1}^{3}\right)\left(E_{2,1}^{1}\right)
+2\left(H_{2,1}^{4}\right)\left(E_{2,1}^{0}\right)\nonumber\\
&=& 2\cdot 0\cdot 0+2\cdot 2\cdot 1 + 2 \cdot 0\cdot 3+2\cdot 0\cdot 1+2\cdot 0\cdot 0= 4, \nonumber\\
\left\langle\mathcal{H}^{0}\oplus \mathcal{E}^{2}\right\rangle_{2,1}^{3} 
 &=&
2\left(E_{2,1}^{0}E_{2,1}^{3}\right)+
2\left(E_{2,1}^{1}E_{2,1}^{2}\right)= 2\cdot 0\cdot 1+2\cdot 1\cdot 3=6,\nonumber\\
\left\langle\mathcal{H}^{1}\oplus \mathcal{E}^{1}\right\rangle_{2,1}^{3} 
 &=& 
 2\left(H_{2,1}^{0}\right)\left(E_{2,1}^{3}\right)+
 2\left(H_{2,1}^{1}\right)\left(E_{2,1}^{2}\right)+
 2\left(H_{2,1}^{2}\right)\left(E_{2,1}^{1}\right)+
 2\left(H_{2,1}^{3}\right)\left(E_{2,1}^{0}\right)\nonumber\\
 &=&2\cdot 0\cdot 1+2\cdot 2\cdot 3+2\cdot 0\cdot 1+2\cdot 0\cdot 0= 12, \text{ and}\nonumber\\
\left\langle\mathcal{H}^{2}\oplus \mathcal{E}^{0}\right\rangle_{2,1}^{3} 
&=&2\left(H_{2,1}^{0}H_{2,1}^{3}\right)+
2\left(H_{2,1}^{1}H_{2,1}^{2}\right)= 2\cdot 0\cdot 0+2\cdot 2\cdot 0=0\nonumber.
\end{eqnarray}
Hence, $$N(2,2,4)=11+4+6+12+0=33.$$

Since the inputs are relatively small, we verified $N(2,2,4)=33$ using Theorem \ref{bigtheorem}; however, it is simpler in this case to observe that $\binom{7}{4}=35$ and exactly two subsets of size $4$ fail to be power dominating sets, as illustrated in Figure 2.

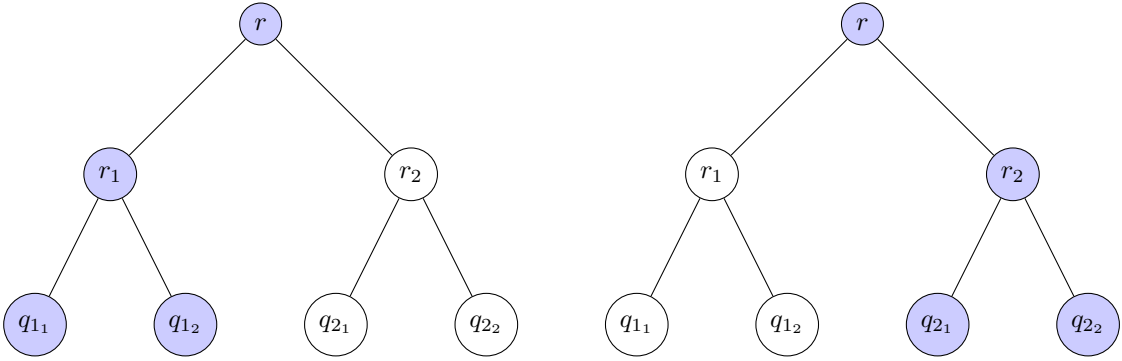
\begin{figure}[h]
\caption{These two subsets of size 4 are the only ones that fail to be power dominating sets.}
\begin{tikzpicture}[auto=center,every node/.style={circle,fill=white}]
        \node[draw, fill=blue!20] (r)     at (0,4) {$r$};
        \node[draw, fill=blue!20] (r1)    at (-2,2) {$r_1$};
        \node[draw] (r2)    at (2,2) {$r_2$};
        \node[draw, fill=blue!20] (q11)   at (-3,0) {$q_{1_1}$};
        \node[draw, fill=blue!20] (q12)   at (-1,0) {$q_{1_2}$};
        \node[draw] (q21)   at (1,0) {$q_{2_1}$};
        \node[draw] (q22)   at (3,0) {$q_{2_2}$};
        \draw (r) -- (r1);
        \draw (r) -- (r2);
        \draw (r1) -- (q11);
        \draw (r1) -- (q12);
        \draw (r2) -- (q21);
        \draw (r2) -- (q22);
        \node[draw, fill=blue!20] (rr)    at (8,4) {$r$};
        \node[draw] (rr1)   at (6,2) {$r_1$};
        \node[draw, fill=blue!20] (rr2)   at (10,2) {$r_2$};
        \node[draw] (qq11)  at (5,0) {$q_{1_1}$};
        \node[draw] (qq12)  at (7,0) {$q_{1_2}$};
        \node[draw, fill=blue!20] (qq21)  at (9,0) {$q_{2_1}$};
        \node[draw, fill=blue!20] (qq22)  at (11,0) {$q_{2_2}$};
        \draw (rr) -- (rr1);
        \draw (rr) -- (rr2);
        \draw (rr1) -- (qq11);
        \draw (rr1) -- (qq12);
        \draw (rr2) -- (qq21);
        \draw (rr2) -- (qq22);
    \end{tikzpicture}
\end{figure}
\end{example}

The previous computation can be a bit cumbersome when done by hand and therefore it is desirable to use a computer.  Assuming constant time for arithmetic operations, $$ \binom{m}{\ell}\binom{\ell}{s_0, s_1, \ldots, s_k}\binom{m-\ell}{t_0, t_1, \ldots, t_k}\left(\prod\limits_{1\leq j\leq \ell} H_{m,h}^{i_j}\right)\left(\prod\limits_{\ell< j\leq m} E_{m,h}^{i_j}\right)$$ is computed in linear time relative to $k$ (since $m<k$).  The complexity of computing $\left\langle\mathcal{H}^{\ell}\oplus \mathcal{E}^{m-\ell}\right\rangle_{m,h}^{k}$ is increased by the complexity of listing the partitions of $k$ and keeping the ones that are the concatenation of two non-increasing sequences.  This can be accomplished by first constructing a tree of all non-decreasing partitions of $k-i$ $(0\leq i\leq k)$ into $j\leq m$ parts by using the constructive recurrence $p(k,m)=p(k-1,m-1)+p(k-m,m)$ which uses less than $2^{k+1}$ iterations since the complete binary tree has $2^{k+1}-1$ vertices. 

It follows that the computational complexity to determine $\left\langle\mathcal{H}^{\ell}\oplus \mathcal{E}^{m-\ell}\right\rangle_{m,h}^{k}$ is $O(m)\cdot (2^{k+1})=2^{O(k)}$.  Since $N(m,h,k)$ requires only $2m+1$ computations equivalent to $\left\langle\mathcal{H}^{\ell}\oplus \mathcal{E}^{m-\ell}\right\rangle_{m,h}^{k}$ (and some constant time arithmetic operations) we have that $N(m,h,k)$ is computed in $2^{O(k)}$ time (exponential with linear exponent).

\section{Future Directions}\label{future}
Power domination in graphs arose from the phase measurement unit problem where it is desirable to place the least number of monitors in an electrical power system.  The desire for minimality of monitors in this problem is motivated by the expense of purchasing and placing such monitors in the grid.  The work on this paper was motivated by considering the additional time cost of determining the minimum number of monitors needed as well as an optimal placement of those monitors.  This problem is known to be to be NP-complete even for planar bipartite graphs (\cite{brueni2005pmu}).  This inspires the following questions.
\begin{question}
For fixed $m,h\geq 2$, what is the minimum value of $k_{\alpha}$ such that  $p(m,h,k_\alpha)\geq \alpha$?  For fixed $m\geq 2$ do the following limits exist $$\lim_{h\rightarrow \infty}\left\{ \left.\min_{k\geq 1}\left\{\frac{k}{n}\right\} \right\vert  p(m,h,k)=1\right\}$$ where $n$ is the number of vertices in the $m$-ary tree?
\end{question}
\begin{question}
Is it possible to compute $N(m,h,k)$ in polynomial time? If so, is there a pattern to optimal placements of the $N(m,h,k)$ monitors?
\end{question}
Determining a closed-form formula for $N(m,h,k)$ (or at least its asymptotic behaviors) would be beneficial in answering the previous questions.  The sequence $\left\{\sum\limits_{k}N(2,h,k)\right\}_{h\geq 1}=\{1,7,94,19192,\ldots\}$ does not appear in the Online Encyclopedia of Integer Sequences (\cite{oeis}).
\begin{question}
Does $\left\{\sum\limits_{k}N(m,h,k)\right\}_{h\geq 1}$ count anything other than power domination sets of any size for a complete $m$-ary tree of height $h$ (when $m\geq 2$)? 
\end{question}



\end{document}